\documentclass{article}

\usepackage[utf8]{inputenc} \usepackage[T1]{fontenc}    \usepackage[english]{babel}

\usepackage{amsfonts}
\usepackage{nicefrac}       \usepackage{microtype}      \usepackage{lmodern}
\usepackage{amssymb,amsmath,amsthm}
\usepackage{stmaryrd}
\usepackage{bbm,bm}
\usepackage{latexsym}
\usepackage{xcolor}

\usepackage{enumerate}
\usepackage{verbatim}
\usepackage{booktabs}       

\usepackage{url}            \usepackage[colorlinks=true]{hyperref}
\usepackage{cleveref}

\usepackage{parskip}
\usepackage{geometry}
\usepackage[short]{optidef}

\usepackage{thmtools}

\declaretheorem{theorem}
\declaretheorem{corollary}
\declaretheorem{lemma}
\declaretheorem{proposition}
\declaretheorem{observation}
\declaretheorem{fact}

\declaretheoremstyle[qed=$\square$]{definitionwithend}
\declaretheorem[style=definitionwithend]{definition}
\declaretheorem[style=definitionwithend]{assumption}

\declaretheorem[style=definitionwithend]{remark}

\crefname{fact}{Fact}{Facts}
\crefname{algorithm}{Algorithm}{Algorithms}
\crefname{assumption}{Assumption}{Assumptions}

\usepackage{import}
\usepackage{pdfpages}
\usepackage{transparent}

\definecolor{gold}{rgb}{0.85,0.65,0}

\newcommand{\by}{\times}

\newcommand{\ip}[1]{\ensuremath{\left\langle #1 \right\rangle}}

\newcommand{\set}[1]{\left\{#1\right\}}

\newcommand{\bb}{\mathbb}

\def\R{{\mathbb{R}}}
\def\S{{\mathbb{S}}}

\def\X{{\mathbb{X}}}

\def\cF{{\cal F}}
\def\cG{{\cal G}}

\def\cM{{\cal M}}

\def\cR{{\cal R}}
\def\cS{{\cal S}}

\DeclareMathOperator{\Opt}{Opt}

\DeclareMathOperator{\tr}{tr}

\DeclareMathOperator*{\E}{\mathbb{E}}

\DeclareMathOperator{\spann}{span}

\DeclareMathOperator{\rint}{rint}

\DeclareMathOperator{\conv}{conv}
\DeclareMathOperator{\cone}{cone}
\DeclareMathOperator{\clconv}{clconv}
\DeclareMathOperator{\clcone}{clcone}

 \usepackage{soul}
 \usepackage{authblk}

\newcommand{\faceeq}{\trianglelefteq}

\newcommand{\SDP}{\textup{SDP}}
\newcommand{\obj}{{\textup{obj}}}
\usepackage[numbers,sort,compress]{natbib}
\PassOptionsToPackage{numbers, compress}{natbib}

\newcommand{\mathprog}[1]{}
\usepackage{tikz}

\begin{document}

\title{On semidefinite descriptions for convex hulls of quadratic programs}
\author[1,2]{Alex L.\ Wang}
\author[1]{Fatma K\i l\i n\c{c}-Karzan}
\affil[1]{Carnegie Mellon University, Pittsburgh, PA, 15213, USA.}
\affil[2]{Purdue University, West Lafayette, IN, 47907, USA.}
\date{\today}

\maketitle

\begin{abstract}
    Quadratically constrained quadratic programs (QCQPs) are a highly expressive class of nonconvex optimization problems. While QCQPs are NP-hard in general, they admit a natural convex relaxation via the standard semidefinite program (SDP) relaxation. In this paper we study when the convex hull of the epigraph of a QCQP coincides with the projected epigraph of the SDP relaxation. We present a sufficient condition for convex hull exactness and show that this condition is further necessary under an additional geometric assumption. The sufficient condition is based on geometric properties of $\Gamma$, the cone of convex Lagrange multipliers, and its relatives $\Gamma_1$ and $\Gamma^\circ$.
\end{abstract}

\section{Introduction}
\label{sec:intro}

A quadratically constrained quadratic program (QCQP) is an optimization problem of the form
\begin{align}
	\label{eq:qcqp}
\Opt\coloneqq \inf_{x\in\R^n}\set{q_\obj(x):\,
	q_i(x) \leq 0,\,\forall i\in[m]},
\end{align}
where $q_\obj, q_1,\dots,q_m:\R^n\to\R$ are each (possibly nonconvex) quadratic functions.
For each $i\in[0,m]$, we will write $q_i(x) = x^\top A_i x + 2b_i^\top x + c_i$ for $A_i\in\S^n$, $b_i\in\R^n$, and $c_i\in\R$.
These optimization problems arise naturally in applications (see e.g., \cite{wang2020tightness} and references therein). 

Although QCQPs are NP-hard in general, they admit a natural convex relaxation known as the standard semidefinite program (SDP) relaxation, 
\begin{align*}
&\Opt_\SDP \coloneqq\inf_{\substack{x\in\R^n\\X\in\S^n}}\set{\ip{A_\obj, X} + 2b_\obj^\top x + c_\obj:\,
        \begin{array}{l}
	 \ip{A_i,X} + 2b_i^\top x + c_i \leq0,\,\forall i\in[m]\\
     X\succeq xx^\intercal
	 \end{array}}.
\end{align*}
A growing line of work has offered deterministic conditions under which the SDP relaxation of a general QCQP is exact for various definitions of exactness.
In their celebrated paper, \citet{fradkov1979s-procedure} prove the S-lemma, which implies \textit{objective value exactness}---the condition that the optimal value of the QCQP and the optimal value of its SDP relaxation coincide---holds for QCQPs with a single constraint. \citet{locatelli2016exactness,jeyakumar2013trust,bomze2018extended} present sufficient conditions for the same condition to hold under additional linear constraints and/or one additional quadratic constraint (possibly using strengthened relaxations).
\citet{burer2019exact} give sufficient conditions for objective value exactness for general diagonal QCQPs---those QCQPs for which $A_\obj,A_1,\dots,A_m$ are diagonal matrices.
\citet{sojoudi2014exactness} study sparse SDPs where the $A_i$ matrices have an underying graph structure and give sufficient conditions for objective value exactness phrased in terms of the graph structure and the associated coefficients.
\citet{wang2020tightness} continue this line of work by developing a general framework for deriving sufficient conditions for both objective value exactness and \textit{convex hull exactness}---the condition that the convex hull of the QCQP epigraph coincides with the (projected) SDP epigraph---for QCQPs where a specific dual set, $\Gamma_1$, is \emph{polyhedral} (see \cref{sec:prelim}).
Beyond being a natural sufficient condition for objective value exactness, convex hull exactness has its own applications and motivation, particularly
in the derivation of strong relaxations of certain critical substructures in nonconvex problems; see e.g., \cite{wang2020tightness} and references therein. 

While the framework presented in \cite{wang2020tightness} at once covers many existing results on exactness \cite{burer2019exact,fradkov1979s-procedure,hoNguyen2017second,modaresi2017convex,yildiran2009convex,burer2017how}, it is still quite limited. In particular, the assumption that $\Gamma_1$ is polyhedral is rarely satisfied outside of simultaneously diagonalizable QCQPs and precludes the results in \cite{wang2020tightness} from being applicable to a wider range of interesting QCQPs.
In this paper, we generalize the framework of \cite{wang2020tightness} by eliminating its reliance on the polyhedrality assumption.
Specifically, 
\cref{thm:conv_hull_sufficient}
states that convex hull exactness holds as long as certain systems of equations (that depend on $\Gamma$, $\Gamma_1$, or $\Gamma^\circ$) contain nontrivial solutions.
This sufficient condition can be checked systematically if $\Gamma$, $\Gamma_1$, or $\Gamma^\circ$ is sufficiently simple. 
Furthermore, \cref{thm:conv_nec} states that this sufficient condition is also \emph{necessary} under the assumption that $\Gamma^\circ$ is facially exposed (see \cref{as:facially_exposed}).  To the best of our knowledge, this is the first necessary and sufficient condition for convex hull exactness even in the context of diagonal QCQPs (where $\Gamma$ is automatically polyhedral and facially exposed).
In \cref{thm:polyhedral}, we provide refined \emph{necessary and sufficient} conditions for convex hull exactness in the previous polyhedral setting of \cite{wang2020tightness}. 
We illustrate our results (\cref{thm:conv_hull_sufficient,thm:polyhedral}) on two applications (quadratic matrix programs (QMPs) and sets defined by two quadratic constraints) in \cref{sec:apps}. In the first application, the resulting $\Gamma$ set is non-polyhedral, and thus the sufficient conditions from \cite{wang2020tightness} are not applicable. In the second application, we illustrate that our \cref{thm:polyhedral} reduces to a succinct and easy to check necessary and sufficient condition for convex hull exactness when $q_\obj=0$ and $m=2$.

\subsection{Notation}
For a positive integer $n$, let $[n]\coloneqq\set{1,\dots,n}$ and let $[0,n]\coloneqq\set{0,1,\dots,n}$.
For $i\in[n]$, let $e_i$ be the $i$th standard basis vector in $\R^n$.
Let $0_n$ denote the zero vector in $\R^n$.
Let $\S^n$ denote the set of $n\by n$ real symmetric matrices and $\S^n_+$ the cone of positive semidefinite (PSD) matrices.
For $M\in\S^n$, we write $M\succeq 0$ (resp.\ $M\succ 0$) to denote that $M$ is PSD (resp.\ positive definite).
Let $\ker(M)$ denote the kernel of $M$.
Suppose $\cM$ is a subset of some Euclidean space. Let
$\conv(\cM)$,
$\clcone(\cM)$,
$\spann(\cM)$,
$\cM^\perp$,
$\rint(\cM)$,
and
$\dim(\cM)$
denote the
convex hull, closed conic hull, span (linear hull), orthogonal complement, relative interior, and dimension of $\cM$ respectively.
Specifically, if $\E$ is the Euclidean space containing $\cM$, then $\cM^\perp\coloneqq\set{x\in \E:\, \ip{x,m}= 0,\,\forall m\in\cM}$.
Suppose $K$ is a cone in some Euclidean space. Let $K^\circ$ denote the polar cone of $K$. The notation $F\faceeq K$ denotes that $F$ is a face of $K$. By convention, faces of cones are always nonempty.
Given $x,y$ in some Euclidean space $\E$, let $[x\pm y]\coloneqq \conv(\set{x-y, x+y})$. 

\section{Preliminaries}
\label{sec:prelim}

We study the epigraph of the QCQP in \eqref{eq:qcqp}:
\begin{align*}
	\cS \coloneqq \set{(x,t)\in\R^n\times \R:\, \begin{array}
		{l}
		q_\obj(x) \leq 2t\\
		q_i(x) \leq 0,\,\forall i\in[m]
	\end{array}}.
\end{align*}

It is well known in the QCQP literature~\cite{benTal2001lectures,wang2020tightness,fujie1997semidefinite} that the SDP relaxation of a QCQP is equivalent (under standard assumptions) to the dual of the Lagrangian dual. We introduce notation related to Lagrangian aggregation and then state this equivalence formally in \cref{lem:S_SDP}.

Let $q:\R^n\to\R^{1+m}$ be indexed by $[0,m]$ where $q(x)_i = q_i(x)$ for $i\in[0,m]$.
Let $e_\obj, e_1,\dots,e_m$ denote the corresponding unit vectors in $\R^{1+m}$. 
Note that for $(\gamma_\obj,\gamma)\in\R^{1+m}$, the expression
$\ip{(\gamma_\obj,\gamma),q(x)} = \sum_{i=0}^m \gamma_i q_i(x)$ is a quadratic function in $x\in\R^n$.
Define $A(\gamma_\obj,\gamma)\coloneqq \sum_{i=0}^m\gamma_i A_i$. Similarly, define $b(\gamma_\obj,\gamma)$, and $c(\gamma_\obj,\gamma)$.

We will at times work on the slice of $\R^{1+m}$ where the coordinate $\gamma_\obj$ is fixed to $\gamma_\obj = 1$.
Let $A[\gamma]\coloneqq A(1,\gamma)$ and similarly define $b[\gamma]$ and $c[\gamma]$. Set $[\gamma,q(x)] \coloneqq \ip{(1,\gamma),q(x)} = q_\obj(x) + \sum_{i=1}^m \gamma_i q_i (x)$.

Note that
\vspace{-10pt}
\begin{align*}
	\ip{(\gamma_\obj,\gamma),q(x)}
&=x^\top A(\gamma_\obj,\gamma) x + 2b(\gamma_\obj,\gamma)^\top x + c(\gamma_\obj,\gamma),\quad\text{and}\\
[\gamma,q(x)]
&= x^\top A[\gamma] x + 2b[\gamma]^\top x + c[\gamma].
\end{align*}

We extend the following definition from~\cite{wang2020tightness}.
\begin{definition}
\label{def:convex_lagrange_mulipliers}
The \textit{cone of convex Lagrange multipliers} for \eqref{eq:qcqp} and its cross section at $\gamma_0 = 1$ are respectively
\begin{gather*}
	\Gamma \coloneqq \set{(\gamma_\obj,\gamma)\in\R_+\times\R^m_+:\,
		A(\gamma_\obj,\gamma)\succeq 0
    }\\
	\Gamma_1 \coloneqq
\set{\gamma\in\R_+^m:\,
	A[\gamma]\succeq 0
	}. \qedhere
\end{gather*}
\end{definition}
These sets are related to the feasible domain of a \emph{partial} dual of the SDP relaxation (see \cref{rem:gamma_1_is_dual}).
Note that $\ip{(\gamma_\obj,\gamma),q(x)}$ is a convex function in $x$ for any $(\gamma_\obj,\gamma)\in\Gamma$. Similarly, $[\gamma,q(x)]$ is convex in $x$ for any $\gamma\in\Gamma_1$.

We make the following assumption for the remainder of the paper. This assumption can be interpreted as a dual strict feasibility condition and is  standard~\cite{burer2019exact,wang2020tightness}. 

\begin{assumption}
\label{as:definite}
There exists $(\gamma_\obj^*,\gamma^*)\in\Gamma$ such that $A(\gamma_\obj^*,\gamma^*)\succ 0$. 
\end{assumption}

\Cref{as:definite} is equivalent to the assumption that there exists $\gamma^*\in\Gamma_1$ 
such that $A[\gamma^*]\succ 0$.

\begin{remark}
\label{rem:Gamma_generated_by_Gamma_P}
Under \cref{as:definite},
it holds that 
$\Gamma = \clcone\left(\set{(1,\gamma):\, \gamma\in\Gamma_1}\right)$. (See discussion following \cite[Assumption 2]{wang2020tightness}.)\mathprog{\qed}
\end{remark}

The \emph{projected SDP relaxation} of $\cS$ is given by
\begin{align}
\label{eq:sdp_relaxation_primal}
\begin{gathered}
\cS_\SDP\coloneqq\set{(x,t)\in\R^n\times\R:\, \begin{array}
	{l}
	\exists X\succeq xx^\top:\\
	\ip{A_\obj, X} + 2b_\obj^\top x + c_\obj \leq 2t\\
	\ip{A_i, X} + 2b_i^\top x + c_i \leq 0,\,\forall i\in[m]
\end{array}},\qquad
\Opt_\textup{SDP}\coloneqq \inf_{(x,t)\in\cS_\SDP}2t
\end{gathered}
\end{align}
Taking $X = xx^\top$ in \eqref{eq:sdp_relaxation_primal} gives
$\Opt\geq \Opt_\SDP$ and
$\conv(\cS) \subseteq \cS_\SDP$.

The following lemma rewrites $\cS_\SDP$ in terms of $\Gamma$, $\Gamma_1$, and $\Gamma^\circ$ and follows from a straightforward duality argument.
These expressions allow for a straightforward analysis of $\cS_\SDP$ if the corresponding dual set is sufficiently simple.
\begin{restatable}{lemma}{lemSSDP}
\label{lem:S_SDP}
Suppose \cref{as:definite} holds. Then,
\begin{align*}
\cS_\SDP &= \set{(x,t)\in\R^{n+1}:\, [\gamma,q(x)] \leq 2t ,\,\forall \gamma\in\Gamma_1}\\
&= \set{(x,t)\in\R^{n+1}: \,
\ip{(\gamma_\obj,\gamma), q(x)}\leq 2\gamma_\obj t,\, \forall (\gamma_\obj,\gamma)\in\Gamma}\\
& = \set{(x,t)\in\R^{n+1}:\, q(x) - 2te_\obj \in\Gamma^\circ}.
\end{align*}
\end{restatable}

\begin{proof}
    Fix $(x,t)\in\R^{n+1}$. Note that
    \begin{align*}
    &\sup_{\gamma\in\Gamma_1} [\gamma,q(x)]= \inf_{\xi\in\S^n_+} \set{q_\obj(x) + \ip{A_\obj, \xi}:
        q_i(x) + \ip{A_i, \xi} \leq 0 ,\, \forall i\in[m]}.
    \end{align*}
    Here, the second equality follows from the strong conic duality theorem and \cref{as:definite}. Taking $X \coloneqq xx^\top + \xi$, we deduce that the first equality in \cref{lem:S_SDP} holds.
    
    By \cref{as:definite}, $\Gamma = \clcone\set{(1,\gamma):\, \gamma\in\Gamma_1}$ so that $[\gamma,q(x)]\leq 2t$ for all $\gamma\in\Gamma_1$ if and only if $\ip{(\gamma_\obj,\gamma),q(x)}\leq 2\gamma_\obj t$ for all $(\gamma_\obj,\gamma)\in\Gamma$; this gives the second equality. The third equality holds by definition of the polar cone.
\end{proof}

\begin{corollary}
\label{cor:opt_sdp_saddle}
Suppose \cref{as:definite} holds. Then,
$\Opt_\SDP = \inf_{x\in\R^n}\sup_{\gamma\in\Gamma_1} [\gamma,q(x)]$.
\end{corollary}

\begin{corollary}
\label{cor:SSDP_closed}
Suppose \cref{as:definite} holds. Then, $\cS_\SDP$ is closed.
\end{corollary}

\begin{remark}
\label{rem:gamma_1_is_dual}
We can interpret $\Opt_\textup{SDP}$ in \eqref{eq:sdp_relaxation_primal} as minimizing over $x\in\R^n$ the value of an inner minimization problem over the matrix variable $X$. Writing $X = xx^\top + \xi$ and taking the SDP dual in $\xi$ results in the same saddle-point structure observed in \cref{cor:opt_sdp_saddle}. In other words, \emph{$\Gamma_1$ is the feasible domain to this partial dual of \eqref{eq:sdp_relaxation_primal}.}\mathprog{\qed}
\end{remark}

Our results and analysis depend on the geometry of $\Gamma$, $\Gamma^\circ$, and their faces.
We recall properties of convex cones and their faces specialized to our setting. See \cite{barker1981theory} for a general reference. 

\begin{definition}
\label{def:conjugate_face}
For a face $\cG\faceeq\Gamma^\circ$ and $(g_\obj,g)\in\rint(\cG)$, the \emph{conjugate face of $\cG$} is 
\[\cG^\triangle \coloneqq \Gamma \cap \cG^\perp = \Gamma\cap (g_\obj,g)^\perp.\] 
Similarly, define the conjugate face for $\cF\faceeq \Gamma$.\mathprog{\qed}
\end{definition}
\begin{definition}
A face $\cG\faceeq\Gamma^\circ$ is \textit{exposed} if there exists $(\gamma_\obj,\gamma)\in\Gamma$ such that $\cG = \Gamma^\circ \cap (\gamma_\obj,\gamma)^\perp$.\mathprog{\qed}
\end{definition}

We associate faces of $\Gamma$ and $\Gamma^\circ$ to points $(x,t)\in\cS_\SDP$ as follows.
\begin{definition}
\label{def:face_G_F}
Given $(x,t)\in\cS_\SDP$, let $\cG(x,t)\faceeq\Gamma^\circ$ denote the minimal face of $\Gamma^\circ$ containing
$q(x) - 2te_\obj$ and define $\cF(x,t)\coloneqq \cG(x,t)^\triangle$.\mathprog{\qed}
\end{definition}

The next fact follows from \cref{def:face_G_F}.
\begin{fact}
\label{fact:q_in_rint_G}
For $(x,t)\in\cS_\SDP$, we have $q(x) - 2te_\obj \in\rint(\cG(x,t))$ and $\cF(x,t) = \Gamma \cap (q(x) - 2te_\obj)^\perp$.
\end{fact}

\section{Convex hull exactness}
\label{sec:conv_hull_exactness}

This section presents necessary and sufficient conditions for convex hull exactness, i.e., the property that $\conv(\cS) = \cS_\SDP$.
Below, we rephrase convex hull exactness as a question regarding the existence of nontrivial ``rounding directions.''
\begin{lemma}
\label{lem:conv_iff_R_nontrivial}
Suppose \cref{as:definite} holds. Then, $\conv(\cS) = \cS_\SDP$ if and only if for every $(x,t)\in\cS_\SDP\setminus\cS$, there exists a nonzero $(x',t')\in\R^{n+1}$ and $\epsilon>0$ such that 
$[(x,t)\pm\epsilon (x',t')] \subseteq \cS_\SDP
$.
\end{lemma}
\begin{proof}
    Note that $\cS_\SDP$ is a closed convex set containing no lines.
    It is clear that $(0_n,1)$ is a recessive direction of $\cS_\SDP$.
    We claim that $(0_n,1)$ is the only recessive direction: Let $\gamma^*$ such that $A[\gamma^*]\succ 0$ (this is possible by \cref{as:definite}) and consider any $(x',t')$ where $x'$ is nonzero. Then, for any $(\tilde x,\tilde t)\in\cS_\SDP$, we have $2(\tilde t+\alpha t') < [\gamma^*, q(\tilde x+\alpha x')]$ for all $\alpha>0$ large enough.
We deduce by \cite[Theorem 18.5]{rockafellar1970convex}, that $\cS_\SDP$ is the sum of the convex hull of its extreme points and the direction $(0_n,1)$. In particular $\conv(\cS) = \cS_\SDP$ if and only if $(x,t)$ is not extreme for every $(x,t)\in\cS_\SDP\setminus\cS$. By definition, $(x,t)$ is not extreme if and only if there exists a nonzero $(x',t')$ and $\alpha>0$ such that $[(x,t)\pm \alpha (x',t')]\subseteq \cS_\SDP$.
\end{proof}

We capture the relevant set of ``rounding directions'' from \cref{lem:conv_iff_R_nontrivial} in the following definition.
\begin{definition}
    \label{def:R}
The subspace of \emph{rounding directions} at $(x,t)\in\cS_\SDP$ is
\begin{align*}
\cR(x,t) \coloneqq
\set{(x',t')\in\R^{n+1}:\,
\begin{array}{l}
	\exists \epsilon>0:\,\\
    {}[(x,t)\pm \epsilon (x',t')]\subseteq\cS_\SDP
\end{array}
}.
\end{align*}
This set is \emph{nontrivial} if it is not $\set{0}$.\mathprog{\qed}
\end{definition}

Note that $\cR(x,t)$ is a convex cone and $-\cR(x,t) = \cR(x,t)$ so $\cR(x,t)$ is indeed a subspace.

\begin{remark}\label{rem:suffices_check_epigraph_tight}
Suppose $(x,t)\in\cS_\SDP$ and $2t > \sup_{\gamma\in\Gamma_1}[\gamma,q(x)]$. Then, by \cref{lem:S_SDP} there exists $\alpha>0$ such that $[(x,t)\pm \alpha(0_n,1)]\subseteq\cS_\SDP$. Thus, it suffices to verify the condition of \cref{lem:conv_iff_R_nontrivial} for points $(x,t)\in\cS_\SDP\setminus\cS$ for which $2t = \sup_{\gamma\in\Gamma_1}[\gamma,q(x)]$.\mathprog{\qed}
\end{remark}

\subsection{Sufficient conditions for convex hull exactness}
\label{subsec:conv_suff}

In general, $\cR(x,t)$ is difficult to describe explicitly. 
Below, we identify a more manageable subset $\cR'(x,t)\subseteq\cR(x,t)$. This will lead to a sufficient condition for convex hull exactness via \cref{lem:conv_iff_R_nontrivial}.

\begin{definition}
Let $(x,t)\in\cS_\SDP$ and let $\cG(x,t)$ be the minimal face of $\Gamma^\circ$ containing $q(x)-2te_\obj$. Define
\begin{align*}
&\cR'(x,t) \coloneqq \set{(x',t')\in\R^{n+1}:\, 
q(x+\alpha x') - 2(t+\alpha t') e_\obj \in\spann(\cG(x,t)),\forall \alpha\in\R
}. 
\qedhere
\end{align*}
\end{definition}
To motivate $\cR'(x,t)$, note that by \cref{def:R,lem:S_SDP}, $(x',t')\in\cR(x,t)$ if and only if $\exists\epsilon>0$ such that
\begin{align*}
    q(x+\alpha x') -2 (t+ \alpha t')e_\obj\in\Gamma^\circ,\quad\forall \alpha\in[-\epsilon,\epsilon].
\end{align*}
If we assume (incorrectly) that $q(x+\alpha x')-2(t+\alpha t')e_\obj$ is linear in $\alpha$, then
this holds if and only if
\begin{align*}
    q(x+\alpha x') - 2(t+\alpha t')e_\obj \in \spann(\cG(x,t))
\end{align*}
for all $\alpha\in[-\epsilon,\epsilon]$ (in fact, also for all $\alpha\in \R$). This is precisely the definition of $\cR'(x,t)$. Thus, $\cR'(x,t)$ is a ``first-order approximation'' to $\cR(x,t)$ in this sense.
We will later give additional motivation for $\cR'(x,t)$ when $\Gamma^\circ$ is facially exposed.

Below, we show that $\cR'(x,t)\subseteq \cR(x,t)$.
\begin{lemma}
\label{lem:R'_subset}
Suppose \cref{as:definite} holds and $(x,t)\in\cS_\SDP$. Then, $\cR'(x,t)\subseteq\cR(x,t)$.
\end{lemma}
\begin{proof}
Let $(x,t)\in\cS_\SDP$ and $(x',t')\in\cR'(x,t)$. By continuity and the fact that $q(x) - 2te_\obj \in \rint(\cG(x,t))$ (by \cref{fact:q_in_rint_G}), there exists $\alpha>0$ such that for all $\epsilon\in[\pm\alpha]$ we have 
$ q(x+ \epsilon x') - 2(t + \epsilon t') e_\obj \in\cG(x,t)\subseteq \Gamma^\circ .
$ By the third characterization of $\cS_\SDP$ in  \cref{lem:S_SDP}, we have that $[(x,t)\pm\alpha (x',t')]\subseteq \cS_\SDP$.
\end{proof}

\Cref{lem:conv_iff_R_nontrivial,lem:R'_subset} immediately imply the following sufficient condition for convex hull exactness.
\begin{theorem}
\label{thm:conv_hull_sufficient}
Suppose \cref{as:definite} holds and that for all $(x,t)\in\cS_\SDP\setminus\cS$, the set $\cR'(x,t)$ is nontrivial. Then, $\conv(\cS) = \cS_\SDP$.
\end{theorem}

The condition of \cref{thm:conv_hull_sufficient} on the surface requires checking infinitely many systems of infinitely many constraints.
Nonetheless, it is possible to check these conditions in some situations where $\Gamma$ is sufficiently simple. See \cref{sec:apps}.

\subsection{Necessary conditions for convex hull exactness}
We now show that, under a technical assumption (\cref{as:facially_exposed}), $\cR'(x,t)=\cR(x,t)$ so that the sufficient condition of \cref{thm:conv_hull_sufficient} is also necessary.

\begin{assumption}
\label{as:facially_exposed}
Suppose $\Gamma^\circ$ is facially exposed, i.e., every face of $\Gamma^\circ$ is exposed.
\end{assumption}

This assumption holds for any conic slice of the nonnegative orthant, the second-order cone, or the PSD cone. See \cite{pataki2013connection} for a discussion of this assumption and its connections to the \textit{nice cones}. In general, all \textit{nice cones} are facially exposed.

We will use the following property of exposed faces.
\begin{fact}[\cite{barker1981theory}]
\label{fact:exposed_double_conjugate}
A face $\cG\faceeq\Gamma^\circ$ is exposed if and only if $\cG = (\cG^\triangle)^\triangle$.
\end{fact}

We provide further intuition to the definition of $\cR'(x,t)$ under \cref{as:facially_exposed}. Suppose $(x',t')\in\cR(x,t)$.
Noting that $\cF(x,t)$ captures the convex quadratic inequalities defining $\cS_\SDP$ that are \emph{tight} at $(x,t)$, we deduce that $\ip{(f_\obj,f), q(x+\alpha x') - 2(t + \alpha t') e_\obj}$ must be zero for all $(f_\obj, f)\in\cF(x,t)$ and all small enough $\alpha$. In other words, a necessary condition for $(x',t')\in\cR'(x,t)$ is that
\begin{align*}
    q(x+\alpha x') - 2(t + \alpha t') e_\obj \in \Gamma \cap \cF(x,t)^\perp
\end{align*}
for all small enough $\alpha$. Under, \cref{as:facially_exposed} we have that $\Gamma \cap \cF(x,t)^\perp = \cG(x,t)$.
This same reasoning is presented formally in the lemma below.
\begin{lemma}
\label{lem:R'_equal}
Suppose \cref{as:definite,as:facially_exposed} hold and let $(x,t)\in\cS_\SDP$. Then, $\cR'(x,t) = \cR(x,t)$.
\end{lemma}
\begin{proof}
Fix $(x',t')\in\cR(x,t)$. As $\cR(x,t)$ is a convex cone, we may without loss of generality assume that $[(x,t)\pm (x',t')]\subseteq \cS_\SDP$. Our goal is to show that $(x',t')\in\cR'(x,t)$, i.e., that
\begin{align*}
q(x\!+\!\alpha x') - 2(t\!+\!\alpha t')e_\obj \!\in\!\spann(\cG(x,t)),\,\forall \alpha\in\R.
\end{align*}
As each coordinate of this vector is quadratic in $\alpha$, it suffices to show instead that the inclusion holds for all $\alpha\in[-1,1]$.
Indeed, suppose this inclusion has been shown for all $\alpha\in[-1,1]$ and let $\ell\in \spann(\cG(x,t))^\perp$. Then, $\ip{\ell, q(x+\alpha x') - 2(t+\alpha t')e_\obj}$ is a quadratic function in $\alpha$ mapping the interval $[-1,1]$ to $0$. Thus, it must map all of $\R$ to zero. As $\ell\in \spann(\cG(x,t))^\perp$ was arbitrary, we deduce that $q(x+\alpha x') - 2(t+\alpha t')e_\obj \in\spann(\cG(x,t))$ for all $\alpha\in\R$.

Recall that $\cF(x,t)= \Gamma \cap (q(x) - 2t e_\obj)^\perp$ and let $(f_\obj, f)\in\rint(\cF(x,t))$. By \cref{as:facially_exposed,fact:exposed_double_conjugate}, we may write $\cG(x,t) = \Gamma^\circ \cap (f_\obj, f)^\perp$.
As $[(x,t)\pm (x',t')]\subseteq \cS_\SDP$, we have that $q(x+\alpha x') - 2(t+\alpha t')e_\obj\in \Gamma^\circ$ for all $\alpha \in[-1,1]$. It remains to verify that the map
\begin{align*}
\alpha\mapsto
\ip{
(f_\obj,f),
q(x+\alpha x') - 2(t+\alpha t')e_\obj}
\end{align*}
evaluates to zero on $\alpha \in[-1,1]$.
Again, as $[(x,t)\pm (x',t')]\subseteq \cS_\SDP$, this map is nonpositive for all $\alpha\in[-1,1]$.
Next, note that $(f_\obj,f)\in\cF(x,t) = \Gamma \cap (q(x) - 2te_\obj)^\perp$ so that this map evaluates to zero at $\alpha = 0$. Finally, $(f_\obj,f) \in \Gamma$ implies that this map is also convex. We conclude that this map is identically zero.
\end{proof}

The following necessary and sufficient condition for convex hull exactness then follows from \cref{lem:R'_equal}.
\begin{theorem}
\label{thm:conv_nec}
Suppose \cref{as:definite,as:facially_exposed} hold. Then, $\conv(\cS) = \cS_\SDP$ if and only if for all $(x,t)\in\cS_\SDP\setminus\cS$, the set $\cR'(x,t)$ is nontrivial.
\end{theorem}

\subsection{Explicit descriptions of $\cR'(x,t)$}

\cref{lem:R'_if_epigraph_tight,prop:R'_exposed} below give alternate descriptions of $\cR'(x,t)$ (see \cref{lem:R'_if_epigraph_tight,prop:R'_exposed})
that may be useful in verifying the conditions of \cref{thm:conv_hull_sufficient,thm:conv_nec}.
See, for example, their use in \cref{sec:apps}.
We begin by making an observation and then proving \cref{lem:R'_if_epigraph_tight,prop:R'_exposed}.

\begin{observation}
\label{obs:epi_tight_clcone}
Suppose \cref{as:definite} holds. Let $(x,t)\in\cS_\SDP$ where $2t = \sup_{\gamma\in\Gamma_1} [\gamma,q(x)]$. Then, $\spann(\cG(x,t))\not\supseteq \R\times 0_m$.
In particular, $\cG(x,t)^\perp = \spann\left(\cG(x,t)^\perp \cap \set{\gamma_\obj = 1}\right)$.
\end{observation}
\begin{proof}
    Suppose $\spann(\cG(x,t))\supseteq \R\times 0_m$ so that $(0_n, 1) \in\cR'(x,t)$. By \cref{lem:R'_subset}, there exists $\alpha>0$ so that  $(x,t - \alpha)\in\cS_\SDP$. This contradicts $2t = \sup_{\gamma\in\Gamma_1} [\gamma,q(x)]$. 
Thus, $\spann(\cG(x,t))\not\supseteq\R\times0_m$. Equivalently, $\cG(x,t)^\perp \not\subseteq 0\times \R^m$ and there exists $(1,\bar\gamma)\in\cG(x,t)^\perp$.
    Then, for any $(\gamma_\obj,\gamma)\in\cG(x,t)^\perp$, we can write $(\gamma_\obj,\gamma)$ as a linear combination of
    $(\gamma_\obj,\gamma) + (1-\gamma_\obj) (1,\bar\gamma)$
    and
    $(1,\bar\gamma)$.
    \end{proof}

\begin{proposition}
\label{lem:R'_if_epigraph_tight}
Suppose \cref{as:definite} holds and $(x,t)\in\cS_\SDP$ with $2t=\sup_{\gamma\in\Gamma_1}[\gamma,q(x)]$. Then,
\begin{align*}
\cR'(x,t) &= \set{(x',t')\in\R^{n+1}:\,\begin{array}
	{l}
	(x')^\top A[\gamma] x' = 0  ,\,\forall (1,\gamma)\in\cG(x,t)^\perp\\
	\ip{A[\gamma]x + b[\gamma], x'} - t' = 0,\,\forall (1,\gamma)\in\cG(x,t)^\perp
\end{array}}.
\end{align*}
\end{proposition}
\begin{proof}
    Note that $(x',t')\in\cR'(x,t)$ if and only if for all $(\gamma_\obj,\gamma)\in\cG(x,t)^\perp$, the expression
    \begin{align*}
    &\ip{(\gamma_\obj,\gamma), q(x + \alpha x') - 2(t+\alpha t')e_\obj}\\
    &\quad = 
    \alpha^2 (x')^\top A(\gamma_\obj,\gamma) (x') \\
    &\qquad + 2\alpha\left(\ip{A(\gamma_\obj,\gamma) x + b(\gamma_\obj,\gamma), x'}- t'\right) \\
    &\qquad + \ip{(\gamma_\obj,\gamma), q(x) - 2te_\obj}
    \end{align*}
    is identically zero in $\alpha$. This occurs if and only if for all $(\gamma_\obj,\gamma)\in\cG(x,t)^\perp$,
    \begin{align*}
    &(x')^\top A(\gamma_\obj,\gamma) x' = 0,\quad\text{and}\\& \ip{A(\gamma_\obj,\gamma) x + b(\gamma_\obj,\gamma), x'} - t' = 0.
    \end{align*}
    The result then follows from \cref{obs:epi_tight_clcone}.
    \end{proof}

\begin{proposition}
\label{prop:R'_exposed}
	Suppose \cref{as:definite,as:facially_exposed} hold. Let $(x,t)\in\cS_\SDP$ where $2t = \sup_{\gamma\in\Gamma_1}[\gamma,q(x)]$. Let $(1,f)\in\rint(\cF(x,t))$. Then,
	\begin{align*}
	&\cR'(x,t) = \\
	&\set{(x',t')\in\R^{n+1}:\, \begin{array}
		{l}
		x' \in\ker(A[f])\\
		\ip{A[\eta]x + b[\eta],x'} - t' = 0 ,\,\forall (1,\eta)\in\cG(x,t)^\perp
	\end{array}}.
	\end{align*}
\end{proposition}
\begin{proof}
    Let $(1,f)\in\rint(\cF(x,t))$.
    By \cref{lem:R'_if_epigraph_tight}, it suffices to prove $x'\in\ker(A[f])$ if and only if
    \begin{align*}
    (x')^\top A[\gamma] x' = 0,\,\forall (1,\gamma)\in\cG(x,t)^\perp.
    \end{align*}
    The reverse direction holds immediately as $(1,f) \in \cF(x,t) \subseteq \cG(x,t)^\perp$ and $A[f]\succeq 0$.
    To see the forward direction: Let $x'\in \ker(A[f])$ and set $v_\obj = (x')^\top A_\obj x'$. Similarly, set $v_i = (x')^\top A_i x'$. Then, $\forall (\gamma_\obj,\gamma)\in\Gamma$, we have
    \begin{align*}
    \ip{(v_\obj,v),(\gamma_\obj,\gamma)} &= (x')^\top A(\gamma_\obj,\gamma) x' \geq 0.
    \end{align*}
    Thus, $(-v_\obj,-v)\in \Gamma^\circ$. On the other hand, $\ip{(v_\obj,v),(1,f)} = (x')^\top A[f] x' = 0$. We deduce that $(-v_\obj,-v) \in \Gamma^\circ \cap (1,f)^\perp = \cF(x,t)^\triangle = \cG(x,t)$. In particular, $(x')^\top A(\gamma_\obj,\gamma) x' = \ip{(v_\obj, v), (\gamma_\obj,\gamma)}= 0$ for all $(\gamma_\obj,\gamma)\in\cG(x,t)^\perp$.
    \end{proof}

\subsection{Revisiting the setting of polyhedral $\Gamma$}
\label{subsec:polyhedral}

\citet{wang2020tightness} give sufficient conditions for convex hull exactness under the assumption that $\Gamma_1$ is polyhedral (equivalently, by \cref{rem:Gamma_generated_by_Gamma_P}, $\Gamma$ is polyhedral). This assumption holds, for example, when the set of quadratic forms $\set{A_\obj,A_1,\dots,A_m}$ is simultaneously diagonalizable.
Specializing \cref{thm:conv_nec} to this setting, we 
prove the following \emph{necessary and sufficient} counterpart to \cite[Theorem 1]{wang2020tightness}.

\begin{restatable}
	{theorem}{thmpolyhedral}
	\label{thm:polyhedral}
	Suppose \cref{as:definite} holds and that $\Gamma$ is polyhedral.
	Then, for any $(x,t)\in\cS_\SDP$ such that $2t = \sup_{\gamma\in\Gamma_1}[\gamma,q(x)]$  and any vector $f$ such that $(1,f)\in \rint(\cF(x,t))$, we have
	\begin{align*}
	\cR'(x,t) = \set{(x',t')\in\R^{n+1}:\,\begin{array}{l}
		x' \in \ker(A[f])\\
		\ip{b[\gamma], x'} - t' = 0,\,\forall (1,\gamma) \in \cF(x,t)
	\end{array}}.
	\end{align*}
The only dependence on $(x,t)$ in the above description of $\cR'$ is in the definition of $\cF = \cF(x,t)$. 
	Hence, $\conv(\cS) = \cS_\SDP$ if and only if the above set is nontrivial for every $\cF\faceeq\Gamma$ which is exposed by some vector $q(x) - 2te_\obj$ for $(x,t)\in\cS_\SDP\setminus\cS$.
\end{restatable}
\begin{proof}
We begin by noting that when $\Gamma$ is polyhedral, so too is $\Gamma^\circ$ so that \cref{as:facially_exposed} holds.
Next, we claim that for every face $\cG\faceeq\Gamma^\circ$
we have $\cG^\perp = \spann(\cG^\triangle)$. By definition, $\spann(\cG^\triangle) = \spann(\Gamma\cap \cG^\perp) \subseteq\cG^\perp$.
On the other hand, as $\Gamma$ and $\Gamma^\circ$ are polyhedral, we have that
$\dim(\cG) + \dim(\cG^\triangle) = m$~\cite[Theorem 3]{tam1976note}.
Rearranging this equation, we have $\dim(\cG^\triangle) = m - \dim(\cG) = \dim(\cG^\perp)$. We conclude that $\cG^\perp = \spann(\cG^\triangle)$.

Let $(x,t)\in\cS_\SDP$ such that $2t = \sup_{\gamma\in\Gamma_1}[\gamma,q(x)]$ and let $(1,f)\in\rint(\cF(x,t))$. Then, 
\cref{obs:epi_tight_clcone} and \cref{prop:R'_exposed} imply that
\begin{align*}
\cR'(x,t)
&= \set{(x',t')\in\R^{n+1}:\,\begin{array}{l} x'\in \ker(A[f])\\
	\ip{A[\gamma]x + b[\gamma],x'} - t' = 0 ,\,\forall (1,\gamma)\in\cG(x,t)^\perp
\end{array}}\\
&= \set{(x',t')\in\R^{n+1}:\,\begin{array}{l}
	x'\in \ker(A[f])\\
	\ip{A[\gamma]x + b[\gamma],x'} - t' = 0 ,\,\forall (1,\gamma)\in\cF(x,t)
\end{array}}\\
&= \set{(x',t')\in\R^{n+1}:\,\begin{array}{l} 
	x'\in \ker(A[f])\\
	\ip{b[\gamma],x'} - t' = 0 ,\,\forall (1,\gamma)\in\cF(x,t)
\end{array}}.
\end{align*}
Here, the second line follows as $\cG^\perp = \spann(\cG^\triangle)$ holds for every face $\cG\faceeq\Gamma^\circ$ and by definition $\cF(x,t)=\cG(x,t)^\triangle$. The third line follows from the fact that $(1,f)\in\rint(\cF(x,t))$ implies $\ker(A[f])\subseteq\ker(A[\gamma])$ for every $(1,\gamma)\in\cF(x,t)$. 
The result then follows from \cref{thm:conv_nec}.
\end{proof}

\begin{remark}\label{rem:polyhedral}
The main difference between \cref{thm:polyhedral} and \cite[Theorem 1]{wang2020tightness} is that \cref{thm:polyhedral} considers certain (\textit{a fortiori} semidefinite) faces of $\Gamma$ whereas \cite[Theorem 1]{wang2020tightness} imposes a constraint on \emph{every} semidefinite face of $\Gamma$. This idea of restricting the analysis to certain faces of $\Gamma$ was previously investigated by \cite{burer2019exact,locatelli2020kkt} who used it to provide sufficient conditions for objective value exactness.\mathprog{\qed}
\end{remark}

\section{Example applications of \cref{thm:conv_hull_sufficient,thm:conv_nec,thm:polyhedral}}
\label{sec:apps}

We apply the results of \cref{sec:conv_hull_exactness} to two examples.

 \subsection{Quadratic matrix programs}
\label{subsec:quadratic_matrix_programming}

Quadratic matrix programs (QMPs)~\cite{beck2007quadratic,wang2020tightness} are a generalization of QCQPs where the decision variable $x\in\R^n$ is replaced by a decision matrix $X\in\R^{r\times k}$. These problems find a variety of applications, e.g., robust least squares problems and sphere packing problems.
Formally, a QMP is an optimization problem in the variable $X\in\R^{r\times k}$, where the constraints and objective function are each of the form
\begin{align*}
\tr(X^\top \bb A X) + 2\tr(B^\top X) + c
\end{align*}
for some $\bb A\in\S^r$, $B\in\R^{r\times k}$, and $c\in\R$.

Alternatively, letting $x\in\R^n$ (resp.\ $b\in\R^n$) denote the vector formed by stacking the columns of $X$ (resp.\ $B$) on top of each other, we can rewrite the above expression as
\begin{align*}
x^\top (I_k\otimes \bb A)x + 2\ip{b, x} + c.
\end{align*}
Thus, we can view QMPs as the special class of QCQPs where $A_\obj,A_1,\dots,A_m$ are each of the form $I_k\otimes \bb A$ for some $\bb A \in\S^r$.

Below, we establish a sufficient condition for convex hull exactness in this setting.
Specifically, we will show that if \cref{as:definite} holds, for each $i\in[0,m]$, we can write $A_i = I_k \otimes \bb A_i$ for some $\bb A_i\in\S^r$, and $k\geq m$,
then $\cR'(x,t)$ is nontrivial for every $(x,t)\in\cS_\textup{SDP}\setminus\cS$. By \cref{thm:conv_hull_sufficient}, this implies that $\conv(\cS) = \cS_\SDP$.

Fix $(x,t)\in\cS_\SDP\setminus\cS$.
Based on \cref{thm:conv_hull_sufficient,lem:R'_if_epigraph_tight}, our goal is to prove that
\begin{align}
\cR'(x,t) &= \set{ (x',t')\in\R^{n+1}:\,\begin{array}
	{l}
	\quad x'^\top A[\gamma] x' = 0 ,\,\forall (1,\gamma)\in\cG(x,t)^\perp\\
	\quad \ip{A[\gamma]x + b[\gamma], x'} - t' = 0 ,\,\forall (1,\gamma)\in\cG(x,t)^\perp 
\end{array}} 
\label{eq:qmp_R}
\end{align}
is nontrivial.
We claim that it suffices to show how to construct a nonzero $y\in\R^r$ such that
\begin{align}\label{eq:qmp_y}
y^\top \bb A[\gamma] y = 0 ,\,\forall (1,\gamma) \in\cG(x,t)^\perp.
\end{align}
To see that this suffices, note that for any $w\in\R^k$, the vector $x'\coloneqq w\otimes y$ satisfies the first constraint in \eqref{eq:qmp_R} since for $(1,\gamma)\in\cG(x,t)^\perp$, we have
\begin{align*}
(w\otimes y)^\top A[\gamma] (w\otimes y) = (w^\top w)(y^\top \bb A[\gamma] y) = 0.
\end{align*}
Then, $(w\otimes y, t')\in\cR'(x,t)$ if and only if the following relation also holds
\begin{align*}
\ip{A[\gamma] x + b[\gamma], w\otimes y} - t' = 0 ,\,\forall (1,\gamma)\in\cG(x,t)^\perp.
\end{align*}
This is a system of
$\dim(\cG(x,t)^\perp)$-many homogeneous linear equations in the variables $(w,t')\in\R^{k+1}$.
Note that as $\cG(x,t) \ni q(x) - 2te_\obj$, which is nonzero by assumption, we have that $\dim(\cG(x,t)^\perp) \leq m$. As $k+1> m$ by assumption, we deduce that this system has a nontrivial solution. Thus, we conclude that \eqref{eq:qmp_R} is nontrivial if there exists a nonzero $y\in\R^r$ satisfying \eqref{eq:qmp_y}. 

It remains to construct $y$.
By definition of $\cS_\SDP$, there exists $Y\succeq 0$ such that
\begin{align}
\label{eq:system_in_lifted_space}
\begin{cases}
q_\obj(x) + \ip{A_\obj, Y} \leq 2t,\\
q_i(x) + \ip{A_i, Y} \leq 0,\,\forall i\in[m]. \end{cases}
\end{align}
Without loss of generality, we can assume $Y = ({1\over k}I_k) \otimes \bb Y$.
As $(x,t)\notin\cS$, we have that $\bb Y\in\S^r_+\setminus\set{0}$ and so we may pick a nonzero $y\in\R^r$ such that $yy^\top \preceq \bb Y$.
For notational convenience, let $\ell_\obj \coloneqq y^\top \bb A_\obj y$ and $\ell_i \coloneqq y^\top \bb A_i y$ for $i\in[m]$. 
Note that for any $(\gamma_\obj,\gamma)\in\Gamma$, we have $A(\gamma_\obj,\gamma)\succeq 0$, or equivalently $\bb A(\gamma_\obj,\gamma)\succeq 0$. Thus, $yy^\top \preceq \bb Y$ implies that 
$\ip{(\gamma_\obj, \gamma), (\ell_\obj,\ell)} 
=y^\top \bb A(\gamma_\obj,\gamma) y 
\leq \ip{\bb A(\gamma_\obj,\gamma),\bb Y}$. Also, from $Y = ({1\over k}I_k) \otimes \bb Y$ and the relation between the matrices $A_\obj,A_i$ and $\bb A_\obj, \bb A_i$, we have $\ip{\bb A(\gamma_\obj,\gamma),\bb Y}=\ip{A(\gamma_\obj,\gamma),Y}$. 
We deduce that for $(\gamma_\obj,\gamma)\in\Gamma$,
\begin{align*}
&\ip{\begin{pmatrix}
	\gamma_\obj\\
	\gamma
\end{pmatrix},q(x) - 2te_\obj + \begin{pmatrix}
	\ell_\obj\\\ell
\end{pmatrix}} \\
& \leq 
\ip{\begin{pmatrix}
	\gamma_\obj\\
	\gamma
\end{pmatrix},
q(x) - 2te_\obj + \begin{pmatrix}
	\ip{A_\obj, Y}\\
	\left(\ip{A_i, Y}\right)_i
\end{pmatrix}} \leq 0,
\end{align*}
where the last inequality follows from \eqref{eq:system_in_lifted_space} and $(\gamma_\obj,\gamma)\in\Gamma \subseteq\R^{1+m}_+$.
Thus, $q(x) - 2te_\obj + (\ell_\obj,\ell)\in\Gamma^\circ$. Moreover, as $\bb A(\gamma_\obj,\gamma)\succeq 0$, we have 
\begin{gather*}
0 \geq - y^\top \bb A(\gamma_\obj,\gamma) y 
= \ip{\begin{pmatrix}
	\gamma_\obj\\\gamma
\end{pmatrix}, \begin{pmatrix}
	-\ell_\obj\\
	-\ell
\end{pmatrix}},
\end{gather*}
which implies $-(\ell_\obj,\ell)\in\Gamma^\circ$. 
We have shown that $q(x) - 2te_\obj + (\ell_\obj,\ell)$ and $-(\ell_\obj,\ell)$ both lie in $\Gamma^\circ$. Then, as $q(x) -2te_\obj\in\rint(\cG(x,t))$, we deduce that $(\ell_\obj,\ell)\in\spann(\cG(x,t))$.
In particular, $y^\top \bb A(\gamma_\obj,\gamma) y = \ip{(\gamma_\obj, \gamma), (\ell_\obj,\ell)} = 0$ for all $(\gamma_\obj,\gamma)\in\cG(x,t)^\perp$.

\begin{remark}
\label{rem:qmp}
SDP exactness for QMPs was previously studied by \citet{wang2020tightness,beck2007quadratic}.
Specifically, under \cref{as:definite}, \citet{beck2007quadratic} shows that objective value exactness holds if $k\geq m$ and \citet{wang2020tightness} show that convex hull exactness holds if $k \geq m + 2$.
The above calculation strengthens both of these results by showing that under \cref{as:definite} convex hull exactness holds if $k \geq m$.\mathprog{\qed}
\end{remark}

\subsection{Sets defined by two quadratic constraints}
\label{subsec:two_constraint}

Next, consider convex hull exactness for sets defined by two quadratic constraints.
In the case of strict inequalities and under a strict feasibility assumption, \citet{yildiran2009convex} establishes that the convex hull of such a set is given by two second-order cone constraints that can be obtained from the eigenvalues of an associated matrix pencil. Below, we will characterize when this convex hull is equal to its standard SDP relaxation.

Consider a set defined by two quadratic constraints, $\hat \cS\coloneqq \set{x\in\R^n:\, q_1(x)\leq 0, q_2(x)\leq 0}$. To connect such a set to the epigraphical sets studied previously in this paper, we will define $q_\obj(x)\coloneqq 0$.
Then, \cref{as:definite} states that there exists $\hat{\gamma}=(\hat{\gamma}_1,\hat{\gamma}_2)\in\R^2_+$ such that $\hat{\gamma}_1 A_1 + \hat{\gamma}_2 A_2\succ 0$.
The projected SDP relaxation of this set is given by
\begin{align*}
    \hat\cS_\SDP\coloneqq \set{x\in\R^n:\, \begin{array}{l}
    \exists X\succeq xx^\intercal\\
    \ip{A_i, X} + 2 b_i ^\intercal x + c_i \leq 0,\forall i = 1,2
    \end{array}}.
\end{align*}
Our goal is to understand when it holds that $\conv(\hat\cS)= \hat\cS_\SDP$.
Recognizing these sets as the $x$-components of $\cS$ and $\cS_\SDP$, we deduce the equivalence holds if and only if $\conv(\cS) = \cS_\SDP$.

As $A_\obj = 0$, we have that
\begin{align*}
    \Gamma = \set{(\gamma_\obj,\gamma)\in\R^{1+2}_+:\, \gamma_1 A_1 + \gamma_2 A_2\succeq 0}.
\end{align*}
Note that $\Gamma_1$ is a cone in $\R^2$ so that it is necessarily polyhedral. We deduce that $\Gamma$ is also polyhedral and we may apply \cref{thm:polyhedral}.
As $A[\hat{\gamma}]\succ0$, $\Gamma_1$ is full dimensional. 
Let
$\gamma^{(1)}, \gamma^{(2)}\in \R^2_+$ satisfy
$\Gamma_1 = \cone(\set{\gamma^{(1)}, \gamma^{(2)}})$. 
Then, $\Gamma = \cone\left(\set{(0,\gamma^{(1)}),(0,\gamma^{(2)}), e_\obj}\right)$ so that $\Gamma$ has exactly 6 proper faces that we deal with individually below.

Suppose $\cF = \cone(e_\obj)$. Then,
\begin{align*}
    \set{(x',t')\in\R^{n+1}:\, \begin{array}{l}
    x'\in \ker(0)\\
    -t' = 0
    \end{array}} \supseteq \R^n \times \set{0}
\end{align*}
is nontrivial.

Next, suppose $\cF = \cone((0,\gamma^{(i)}))$ for some $i=1,2$ or $\cF = \cone(\set{(0,\gamma^{(1)}),(0,\gamma^{(2)})})$. Note that $\cF$ does not contain any points of the form $(1,f)$. Thus, by \cref{obs:epi_tight_clcone}, $\cF$ is not exposed by any vector $q(x)-2te_\obj$ where $x,t$ satisfy $2t = \sup_{\gamma\in\Gamma_1}[\gamma,q(x)]$. \Cref{thm:polyhedral} then places no constraint on the corresponding $\cR'$.

Finally, suppose $\cF = \cone(\set{e_\obj, (0,\gamma^{(1)})})$. The case $\gamma^{(2)}$ is analogous.
Note that $(1, \gamma^{(1)})\in\rint(\cF)$. The relevant set in \cref{thm:polyhedral} is
\begin{align*}
    \set{(x',t')\in \R^{n+1}:\, \begin{array}{l}
    x'\in \ker(A[\gamma^{(1)}])\\
    \ip{b[\gamma^{(1)}],x'}= 0\\
    t' = 0
    \end{array}}.
\end{align*}
This set is nontrivial if and only if $\ker(A[\gamma^{(1)}])\cap b[\gamma^{(1)}]^\perp$ is nontrivial.
\Cref{thm:polyhedral} requires that this set is nontrivial if $\cF$ is exposed by some vector of the form $q(x)-2te_\obj$ for some $(x,t)\in\cS_{\SDP}\setminus \cS$, i.e.,
if $x\in\hat{\cS}_{\SDP}\setminus \hat{\cS}$ and $t$ satisfy $t=0$ and
\begin{align*}
    [\gamma^{(1)},q(x)]=0\quad\text{and}\quad
    [\gamma^{(2)},q(x)]<0.
\end{align*}
Suppose this system is feasible and $x$ satisfies the above relations. Our goal is to show that $\gamma^{(1)}\in\R^2_{++}$. Suppose for the sake of contradiction that $\gamma^{(1)}=e_1$. The above relations would imply that $0=[\gamma^{(1)},q(x)]=q_1(x)$ and $0>[\gamma^{(2)},q(x)]=\gamma^{(2)}_1 q_1(x) + \gamma^{(2)}_2 q_2(x) =  \gamma^{(2)}_2 q_2(x)$. We see that $\gamma^{(2)}_2\neq 0$, and in fact $\gamma^{(2)}_2>0$ (as $\gamma^{(2)}\in\R^2_+$), and thus $0>q_2(x)$. Hence, $x\in\hat{\cS}$, a contradiction.
In particular, if $\gamma^{(1)}= e_1$, then $\cF$ is \emph{not} exposed by some vector of the form $q(x)-2te_\obj$ with $(x,t)\in\cS_\SDP\setminus \cS$. Similarly, if $\gamma^{(1)}= e_2$, then $\cF$ is \emph{not} exposed by some vector of the form $q(x)-2te_\obj$ with $(x,t)\in\cS_\SDP\setminus \cS$.
Note also that if $\gamma^{(1)}\in\R^2_{++}$ and $[\gamma^{(1)}, q(x)] = 0 > [\gamma^{(2)}, q(x)]$, then $x \in \hat \cS_\SDP \setminus \hat\cS$.

We summarize this discussion below.
\begin{lemma}\label{lem:two_constraints_sdp-exactness}
Suppose there exists $\hat\gamma\in\R^2_+$ such that $A[\hat\gamma]\succ 0$. Let 
$\gamma^{(1)}, \gamma^{(2)}\in \R^2_+$ be generators of $\Gamma_1$. 
If $\gamma^{(1)}\in\R^2_{++}$ and there exists $x\in\R^n$ such that $[\gamma^{(1)},q(x)]=0$ and $[\gamma^{(2)},q(x)]<0$, then we require $\ker(A[\gamma^{(1)}])\cap b[\gamma^{(1)}]^\perp$ is nontrivial. Similarly, if $\gamma^{(2)}\in\R^2_{++}$ and there exists $x\in\R^n$ such that $[\gamma^{(2)},q(x)]=0$ and $[\gamma^{(1)},q(x)]<0$, then we require $\ker(A[\gamma^{(2)}])\cap b[\gamma^{(2)}]^\perp$ is nontrivial. Then, $\conv(\hat \cS)= \hat\cS_\SDP$ if and only if the requirements stated above hold.

In particular, if $q_1$ and $q_2$ are both nonconvex, primal and dual strict feasibility are satisfied (i.e., $\exists\hat x\in\R^n$ such that $q_1(\hat x), q_2(\hat x)< 0$ and
$\exists\hat\gamma\in\R^2_+$ such that $A[\hat\gamma]\succ 0$), then $\conv(\hat \cS)= \hat \cS_\SDP$ if and only if for both $i=1,2$, we have that
\begin{align*}
\ker(A[\gamma^{(i)}])\cap b[\gamma^{(i)}]^\perp  \text{ is nontrivial}.
\end{align*}
\end{lemma}
\begin{proof}
The first claim has already been proved in the preceding discussion under the dual strict feasibility assumption. 
The remainder of the proof further assumes that $q_1$ and $q_2$ are both nonconvex and that there exists $\hat x\in\R^n$ such that $q_1(\hat x),q_2(\hat x) < 0$. Under these additional assumptions, we will show that $\gamma^{(1)}\in\R^2_{++}$ and there exists $x\in\R^n$ such that $[\gamma^{(1)},q(x)] = 0$ and $[\gamma^{(2)}, q(x)]<0$. An identical proof will show that the analogous statement with the roles of $\gamma^{(1)}$ and $\gamma^{(2)}$ interchanged holds. This will complete the proof of the second claim.

First, $\gamma^{(1)},\gamma^{(2)}\in\R^2_{++}$ follows from the assumption that $q_1,\,q_2$ are both nonconvex.
As  $\gamma^{(i)}\in\R^2_{++}$, we deduce that $A[\gamma^{(i)}]$ has a nontrivial kernel for both $i=1,2$.
Note that $[\gamma^{(i)}, q(\hat x)]<0$ for both $i=1,2$.
Now pick a nonzero vector $y\in\ker(A[\gamma^{(2)}])$. Then,
\begin{align*}
    [\gamma^{(2)}, q(\hat x + \alpha y)] = q(\hat x) + 2 \alpha b[\gamma^{(2)}]^\top y 
\end{align*}
is a linear function in $\alpha$ that is negative at $\alpha = 0$ and that is without loss of generality nonincreasing as $\alpha\to \infty$ (otherwise, negate $y$).
On the other hand, $[\hat\gamma, q(\hat x + \alpha y)]$ is a strongly convex quadratic function in $\alpha$ so that it is positive for some $\alpha$ large enough. As $\hat\gamma\in \Gamma_1 = \cone(\set{\gamma^{(1)}, \gamma^{(2)}})$, we deduce that $[\gamma^{(1)}, q(\hat x + \alpha y)]$ is positive for all $\alpha>0$ large enough. As $[\gamma^{(1)}, q(\hat x+\alpha y)]$ is a continuous function of $\alpha$ and $[\gamma^{(1)}, q(\hat x)]<0$, there exists some $\alpha>0$ for which
$[\gamma^{(1)},q(\hat x + \alpha y)] = 0$ and $[\gamma^{(2)}, q(\hat x + \alpha y)]<0$.
\end{proof}

\begin{remark}
\Cref{lem:two_constraints_sdp-exactness} presents an easy to check condition for SDP convex hull exactness in the case of two quadratic constraints. In contrast,  testing SDP exactness via the main result of \cite{yildiran2009convex} requires first the nontrivial task of computing $\clconv(\hat \cS)$ and then checking if it is equal to $\hat \cS_\SDP$. 
\mathprog{\qed}
\end{remark}

\section*{Acknowledgments}
This work was supported by NSF [CMMI 1454548], ONR [N00014-19-1-2321], and AFOSR [FA95502210365].

{
\bibliographystyle{plainnat}

\begin{thebibliography}{19}
\providecommand{\natexlab}[1]{#1}
\providecommand{\url}[1]{\texttt{#1}}
\expandafter\ifx\csname urlstyle\endcsname\relax
  \providecommand{\doi}[1]{doi: #1}\else
  \providecommand{\doi}{doi: \begingroup \urlstyle{rm}\Url}\fi

\bibitem[Barker(1981)]{barker1981theory}
G.~P. Barker.
\newblock Theory of cones.
\newblock \emph{{Linear Algebra Appl.}}, 39:\penalty0 263--291, 1981.

\bibitem[Beck(2007)]{beck2007quadratic}
A.~Beck.
\newblock Quadratic matrix programming.
\newblock \emph{{SIAM J.\ Optim.}}, 17\penalty0 (4):\penalty0 1224--1238, 2007.

\bibitem[Ben-Tal and Nemirovski(2001)]{benTal2001lectures}
A.~Ben-Tal and A.~Nemirovski.
\newblock \emph{Lectures on Modern Convex Optimization}, volume~2 of
  \emph{MPS-SIAM Ser. Optim.}
\newblock SIAM, 2001.

\bibitem[Bomze et~al.(2018)Bomze, Jeyakumar, and Li]{bomze2018extended}
I.~M. Bomze, V.~Jeyakumar, and G.~Li.
\newblock Extended trust-region problems with one or two balls: exact
  copositive and lagrangian relaxations.
\newblock \emph{{J.\ Global Optim.}}, 71\penalty0 (3):\penalty0 551--569, 2018.

\bibitem[Burer and K{\i}l{\i}n\c{c}-Karzan(2017)]{burer2017how}
S.~Burer and F.~K{\i}l{\i}n\c{c}-Karzan.
\newblock How to convexify the intersection of a second order cone and a
  nonconvex quadratic.
\newblock \emph{{Math.\ Program.}}, 162:\penalty0 393--429, 2017.

\bibitem[Burer and Ye(2019)]{burer2019exact}
S.~Burer and Y.~Ye.
\newblock Exact semidefinite formulations for a class of (random and
  non-random) nonconvex quadratic programs.
\newblock \emph{{Math.\ Program.}}, 181:\penalty0 1--17, 2019.

\bibitem[Fradkov and Yakubovich(1979)]{fradkov1979s-procedure}
A.~L. Fradkov and V.~A. Yakubovich.
\newblock The {S}-procedure and duality relations in nonconvex problems of
  quadratic programming.
\newblock \emph{Vestnik Leningrad Univ. Math.}, 6:\penalty0 101--109, 1979.

\bibitem[Fujie and Kojima(1997)]{fujie1997semidefinite}
T.~Fujie and M.~Kojima.
\newblock Semidefinite programming relaxation for nonconvex quadratic programs.
\newblock \emph{{J.\ Global Optim.}}, 10\penalty0 (4):\penalty0 367--380, 1997.

\bibitem[Ho-Nguyen and K{\i}l{\i}n\c{c}-Karzan(2017)]{hoNguyen2017second}
N.~Ho-Nguyen and F.~K{\i}l{\i}n\c{c}-Karzan.
\newblock A second-order cone based approach for solving the {Trust Region
  Subproblem} and its variants.
\newblock \emph{{SIAM J.\ Optim.}}, 27\penalty0 (3):\penalty0 1485--1512, 2017.

\bibitem[Jeyakumar and Li(2014)]{jeyakumar2013trust}
V.~Jeyakumar and G.~Y. Li.
\newblock Trust-region problems with linear inequality constraints: {E}xact
  {SDP} relaxation, global optimality and robust optimization.
\newblock \emph{{Math.\ Program.}}, 147:\penalty0 171--206, 2014.

\bibitem[Locatelli(2016)]{locatelli2016exactness}
M.~Locatelli.
\newblock Exactness conditions for an {SDP} relaxation of the extended trust
  region problem.
\newblock \emph{{Oper.\ Res.\ Lett.}}, 10\penalty0 (6):\penalty0 1141--1151,
  2016.

\bibitem[Locatelli(2020)]{locatelli2020kkt}
M.~Locatelli.
\newblock {KKT}-based primal-dual exactness conditions for the {S}hor
  relaxation.
\newblock \emph{arXiv preprint}, 2011.05033, 2020.

\bibitem[Modaresi and Vielma(2017)]{modaresi2017convex}
S.~Modaresi and J.~P. Vielma.
\newblock Convex hull of two quadratic or a conic quadratic and a quadratic
  inequality.
\newblock \emph{{Math.\ Program.}}, 164:\penalty0 383--409, 2017.

\bibitem[Pataki(2013)]{pataki2013connection}
G.~Pataki.
\newblock On the connection of facially exposed and nice cones.
\newblock \emph{Journal of Mathematical Analysis and Applications},
  400\penalty0 (1):\penalty0 211--221, 2013.

\bibitem[Rockafellar(1970)]{rockafellar1970convex}
R.~T. Rockafellar.
\newblock \emph{Convex Analysis}.
\newblock Number~28 in Princeton Mathematical Series. Princeton University
  Press, 1970.

\bibitem[Sojoudi and Lavaei(2014)]{sojoudi2014exactness}
S.~Sojoudi and J.~Lavaei.
\newblock Exactness of semidefinite relaxations for nonlinear optimization
  problems with underlying graph structure.
\newblock \emph{{SIAM J.\ Optim.}}, 24\penalty0 (4):\penalty0 1746--1778, 2014.

\bibitem[Tam(1976)]{tam1976note}
B.~S. Tam.
\newblock A note on polyhedral cones.
\newblock \emph{J. of the Aust. Math. Soc.}, 22\penalty0 (4):\penalty0
  456--461, 1976.

\bibitem[Wang and K{\i}l{\i}n\c{c}-Karzan(2020)]{wang2020tightness}
A.~L. Wang and F.~K{\i}l{\i}n\c{c}-Karzan.
\newblock On the tightness of {SDP} relaxations of {QCQPs}.
\newblock \emph{{Math.\ Program.}}, 2020.

\bibitem[Y{\i}ld{\i}ran(2009)]{yildiran2009convex}
U.~Y{\i}ld{\i}ran.
\newblock Convex hull of two quadratic constraints is an {LMI} set.
\newblock \emph{IMA J. Math. Control Inform.}, 26\penalty0 (4):\penalty0
  417--450, 2009.

\end{thebibliography}

}

\end{document}